\documentclass[12pt]{article}

\usepackage{amscd,amsmath, amssymb, fancyhdr,
	epsfig,color,bbm,url}
\definecolor{dblue}{rgb}{0,0,.6}

\usepackage{hyperref}

\hypersetup{
	colorlinks   = true,
	citecolor    = magenta,
	linkcolor= blue
}

\numberwithin{equation}{section}

\def\eqref#1{(\ref{#1})}

\newcommand{\Z}{{\Bbb Z}}
\newcommand{\C}{{\Bbb C}}

\newcommand{\R}{{\Bbb R}}

\def\1{\sqrt{-1}\.}

\newcommand{\cntrct}                
{\hspace{2pt}\raisebox{1pt}{\text{$\lrcorner$}}\hspace{2pt}}

\makeatletter

\renewcommand{\bar}{\overline}
\renewcommand{\phi}{\varphi}
\renewcommand{\epsilon}{\varepsilon}

\renewcommand{\d}{\partial}

\usepackage{amsthm}
\theoremstyle{definition}
\newtheorem{theorem}{Theorem}[section]
\newtheorem{lemma}[theorem]{Lemma}
\newtheorem{proposition}[theorem]{Proposition}

\newtheorem{cor}[theorem]{Corollary}
\newtheorem{defin}[theorem]{Definition}

\theoremstyle{remark}


\usepackage{tocloft}

\makeatletter

\@addtoreset{equation}{section} \@addtoreset{footnote}{section} \makeatother

\begin{document}
	\begin{center}
		{\LARGE\bf
			K\"unneth formula for Hessian manifolds}
		\medskip
		\medskip

		Pavel Osipov\footnote{National Research University Higher School of Economics, Russian Federation.}\footnote{The study was supported by the Russian Science Foundation grant No. 25-71-00054, https://rscf.ru/project/25-71-00054/}
	\end{center}

    	\begin{abstract}
We study Dolbeault--Koszul cohomology $H^{p,q}(M)$ of flat affine manifolds. 
We proove a K\"unneth formula
\[
H^{p,q}(M\times N) \cong \bigoplus_{i,j} H^{i,j}(M)\otimes H^{p-i,q-j}(N)
\]
for flat affine manifolds $M,N$ with at least one compact. 
For compact manifolds we also give a proof via Hodge theory on flat affine manifolds, 
analogous to the classical K\"unneth formula for Dolbeault cohomology.

We apply this formula to Hessian manifolds. 
A Hessian metric $g$ defines a class $[g]\in H^{1,1}(M)$, 
and metrics in the same class differ by $D\alpha$ for a closed $1$-form $\alpha$. 
Using the K\"unneth formula we describe all Hessian metrics on products, 
on products with hyperbolic manifolds, 
and on manifolds admitting a flat Riemannian metric.
        \end{abstract}
	
	%
    \tableofcontents
\section{Introduction}

A flat affine manifold $(M,D)$ is a manifold $M$ endowed with a flat affine connection $D$. Consider the space $\Lambda^{p,q}M:=\Lambda^pM\otimes \Lambda^qM$ as the space of differential $p$-forms with coefficients in  the flat vector bundle $(\Lambda^qM,D)$.  Then we have a complex
$$
\ldots \xrightarrow{\d}\Lambda^{p-1, q}M \xrightarrow{\d}  \Lambda^{p,q}M  \xrightarrow{\d}  \Lambda^{p,q+1} M \xrightarrow{\d}\ldots
$$
We denote the cohomology of this complex by $H^{p,q}(M)$.   These cohomology groups were introduced by Koszul (\cite{K}) and are analogs of Dolbeault cohomology. Hence, we call them Dolbeault–Koszul cohomology. These cohomology groups were studied by Shima in the context of Hessian geometry. The $L^2$ version of Dolbeault–Koszul cohomology was studied was studied by Akagawa \cite{A}. 

The present article is devoted to the Künneth formula  for Dolbeault–Koszul cohomology and its applications.
    
	\begin{theorem}[K\"unneth formula for Dolbeault–Koszul cohomology]\label{Kunnth}
        Let $M$ and $N$ be flat affine manifolds, with at least one of 
$M,N$ compact. Then we have 
        $$
        H^{p,q}(M\times N)=\bigoplus_{0\le i\le p, 0\le j\le q} H^{i,j}(M)\otimes H^{p-i,q-j}(N).
        $$
	\end{theorem}
    This Künneth formula follows from a general Künneth formula for the cohomology of sheaves (Theorem \ref{ThSHK}). 

Shima developed a Hodge theory for compact flat affine manifolds that is analogous to the Hodge theory for compact complex manifolds (\cite{Sh2},\cite{Sh3}). Using this theory we prove a Künneth formula for $L^2$-forms on a product of flat affine manifolds, and then yields another proof of the Künneth formula for Dolbeault–Koszul cohomology in the case of two compact manifolds. This proof is analogous to the well‑known proof of the Künneth formula for Dolbeault cohomology (\cite{GH}), which highlights the similarity between flat affine  geometry and complex geometry. 

    A Hessian manifold is a flat affine manifold $(M,D)$ endowed with a Hessian metric $g$, that is, a Riemannian metric locally of the form  $g=D^2 f$. Since any K\"ahler metric can be locally expressed as a complex Hessian of a function, we can consider Hessian geometry as an analog of K\"ahler geometry. This continues the parallel between flat affine and complex geometry. For the study of Hessian manifolds, see \cite{Sh1},\cite{Sh2},\cite{Sh3},\cite{ShY},\cite{Y}.

    Compact Hessian manifolds are relatively little studied. Every compact Hessian manifold is the quotient of a convex domain $U\subset \R^n$ by a discrete subgroup of affine automorphisms (\cite{Sh1}). If the Hessian metric can be written as $D\alpha$ for a globally defined closed 1‑form $\alpha$ then the corresponding domain contains no straight lines (\cite{K}). Manifolds arising as quotients of such domains are called hyperbolic flat affine manifolds. If a compact manifold is the quotient of a convex domain $U$ by a discrete subgroup of automorphisms, then $U$ must be a cone (\cite{V}). Two very recent preprints have appeared on the topology of Hessian manifolds (\cite{G}, \cite{L}). They prove that the Euler characteristic of a compact Hessian manifold vanishes, and they classify the topological types of Hessian manifolds  in low dimensions. Let us also mention works on the Monge–Ampère equation on compact Hessian manifolds \cite{ChY}, \cite{GT}, \cite{JO}.

    Although Dolbeault–Koszul cohomology was defined long ago, only a few of its applications are known. Koszul proved that for a compact manifold $H^{1,1}(M)=0$ if and only if $M$ is a hyperbolic flat affine manifold, and deduced that the classical Hopf manifold does not admit a Hessian metric. 
    Shima proved a vanishing theorem analogous to that of Kodaira-Nakano (\cite{Sh2},\cite{Sh3}). Akagawa proved the $L^2$-version of the vanishing theorem.

    We apply Dolbeault-Koszul cohomology to the problem of describing all Hessian metrics on a flat affine manifold. A Riemannian metric is Hessian if and only if it is $\d$-closed as an element of $\Lambda^{1,1}M$. Moreover, for any $\alpha\in \Lambda^1M$ we have $\d(1\otimes\alpha)=D\alpha$ (\cite{Y}). Therefore, any Hessian metric $g$ determines a class $[g]\in H^{1,1}(M)$ and two metrics $g_1$ and $g_2$ belong to the same class if and only if we have $g_1-g_2=D\alpha$ for some 1-form $\alpha$. Note that if $g$ is a Hessian metric, then $g+D\alpha$ is a Hessian metric too, for any sufficiently small 1-form $\alpha$. Thus, it is enough to describe classes in $H^{1,1}(M)$ that represent positive definite symmetric forms.

    The K\"unneth formula allows us to calculate $H^{1,1}(M)$ in some cases, which helps to describe all Hessian metrics.

 For simplicity, we implicitly identify forms on $M$ and $N$ with their pullbacks to $M\times N$.
        \begin{theorem}
        Let $M$ and $N$ be Hessian manifolds, with at least one of 
$M,N$ compact. Then any Hessian metric $g$ on  $M\times N$ is equal to
        $$
        g=g_M+g_N+D \alpha,
        $$
        where $g_M, g_N$ are Hessian metrics on $M,N$ and $\alpha$ is a closed 1-form on $M\times N$.
        \end{theorem}

        The cohomology $H^{1,1}(M)$ of a compact flat affine manifold is zero if and only if the universal covering of $M$ is a domain in $\mathbb{R}^n$ without straight lines. Such  flat affine manifolds are called hyperbolic .

        \begin{cor}
            Let $M$ be a Hessian manifold and $N=C/\Gamma$ a compact hyperbolic flat affine manifold. Then any Hessian metric $g$ on $M\times N$ is equal to
            $$
        g=g_M +D\alpha,
            $$
            where $g_M$ is a Hessian metric on $M$ and $\alpha$ is a closed 1-form on $M\times N$.
        \end{cor}
        \begin{theorem}
            Let $T^n=\R^n/\Z^n$ be a flat torus. Then
            $$
            H^{*,*}(T^n)=\R[a^1,\ldots,a^n,b^1,\ldots,b^n]/_{\left(a^1\right)^{2},\ldots, \left(b^n\right)^{2}},
            $$
            where $a^i=[dx^i\otimes 1]\in H^{1,0} (T^n), b^i=[1\otimes dx^i]\in H^{0,1} (T^n)$.
        \end{theorem}


                \begin{theorem}
           Let $(M,D)$ be a compact flat affine manifold that admits a $D$-flat Riemannian metric. Then any Hessian metric on $M$ is equal to 
            $$
            g=g_{\text{flat}} +D\alpha
            $$
            where $g_{\text{flat}}$ is a flat Riemannian metric on $M$ and $\alpha$ is a closed 1-form.
        \end{theorem}

        \section{Dolbeault-Koszul cohomology}
Let $(M,D)$ be a flat affine manifold. Denote 
$$
\Lambda^{p,q}M=\Lambda^{p}M\otimes \Lambda^qM.
$$
The connection $D$ induces a connection on   $\Lambda^{p,q}M$.
$$
D : \Lambda^{p,q}M \to \Lambda^1 M\otimes \Lambda^{p,q}M. 
$$
Let
$$
e : \Lambda^1 M\otimes \Lambda^{p,q}M\to \Lambda^{p+1,q}M, \ \ \ e(\eta\otimes\alpha\otimes\beta)=(\eta\wedge \alpha)\otimes\beta
$$
and
$$
\bar e : \Lambda^1 M\otimes \Lambda^{p,q}M \to \Lambda^{p,q+1}M  \ \ \ \bar e(\eta\otimes\alpha\otimes\beta)= \alpha\otimes(\eta\wedge \beta)
$$
the exterior multiplication on the first and second factors, respectively.

Denote 
$$
\d=e\circ D : \Lambda^{p,q}M\to \Lambda^{p+1,q}M, \ \ \ \bar \d=\bar e\circ D : \Lambda^{p,q}M\to \Lambda^{p,q+1}M
$$

Let $x^1,\ldots, x^n$ be flat coordinates
where $f_i$ are smooth functions. Then 
$$
\d\left(\sum_{I,J} f_{I,J} dx^I\otimes dx^J\right)=\sum_{I,J} (df_{I,J}\wedge dx^I) \otimes dx^J
$$
$$
\bar \d\left(\sum_{I,J} f_{I,J} dx^I\otimes dx^J\right)=\sum_{I,J}  dx^I \otimes (df_{I,J}\wedge dx^J)
$$
It is easy to check that $\d^2=0, \bar \d^2=0$.
\begin{defin}
    We define the $(p,q)$-cohomology $H^{p,q}(M)$ of a flat affine manifold $(M,D)$ by
    $$
    H^{p,q}(M)=\frac{\operatorname{ker} \left( \d :\Lambda^{p,q}M\to \Lambda^{p+1,q}M\right)}{\operatorname{im} \left( \d :\Lambda^{p-1,q}M\to \Lambda^{p,q}M\right)}.
    $$
\end{defin}
Denote by $P^q$ the sheaf of flat $q$-forms.
\begin{theorem}[\cite{Sh2}]
    We have $H^{p,q}(M)=H^p(M,P^q)$
\end{theorem}
\begin{proof}
    The sequence
    $$
    0\to P^q\xrightarrow{} \Lambda^{0,q}\xrightarrow{\d} \Lambda^{1,q}\xrightarrow{\d}\Lambda^{2,q}\xrightarrow{\d}\ldots
    $$
    is a fine resolution of $P^q$. Therefore, $H^{p,q}(M)=H^p(M,P^q)$.
\end{proof}
\section{K\"unneth formula for Dolbeault–Koszul cohomology}
\begin{theorem}[\cite{D}, Chapter IV, Theorem 15.10]\label{ThSHK}
    Let $\mathcal{A}$ and $\mathcal{B}$ be sheaves of $R$-modules over topological spaces $X$ and $Y$. Assume that $\mathcal{A}$ is torsion free, that $Y$ is compact and that either $X$ is compact or the cohomology groups $H^q(Y, \mathcal{B})$ are finitely generated $R$-modules. There is a split exact sequence
\[
0 \to \bigoplus_{p+q=l} H^p(X, \mathcal{A}) \otimes H^q(Y, \mathcal{B}) \xrightarrow{\mu} H^l(X \times Y, \pi_1^*\mathcal{A}\otimes\pi_2^*\mathcal{B})
\]
\[
\to \bigoplus_{p+q=l+1}  \text{Tor}_1 \left(H^p(X, \mathcal{A}),(H^q(Y, \mathcal{B})\right) \to 0
\]
where $\mu$ is the map given by the cartesian product $\bigoplus \alpha_p \otimes \beta_q \mapsto \sum \alpha_p \times \beta_q$.
\end{theorem}
\begin{cor}\label{sheafkunneth}
    Let $\mathcal{A}$ and $\mathcal{B}$ be locally constant sheaves of vector spaces over manifolds $M$ and $N$, $\pi_1$ and $\pi_2$ be the projections of $M\times N$ on $M$ and $N$. If at least one of the manifolds $M,N$ is compact then we have
        $$
        H^p(M\times N,\pi_1^*\mathcal{A}\otimes\pi_2^*\mathcal{B} )=\bigoplus_{0\le i\le p} H^{i}(M,\mathcal{A})\otimes H^{p-i}(N,\mathcal B).
        $$
\end{cor}
\begin{proof}[Proof of Theorem \ref{Kunnth}]
Sheaves of flat differential forms are locally constant sheaves of vector spaces. Therefore, we can apply Theorem \ref{sheafkunneth}
\begin{multline*}
H^p(M\times N,P^q)=H^p(M\times N,\bigoplus_{0\le j\le n} \left(\pi_1^*P^j\otimes\pi_2^*P^{q-j})\right)\\=\bigoplus_{0\le i\le p, 0\le j\le q} H^{i,j}(M)\otimes H^{p-i,q-j}(N).
\end{multline*}

\end{proof}

    

\section{Hodge theory}
Consider a Riemannian metric $g$ on a flat affine manifold $(M,D)$. Then $g$ induces a Riemannian metric on $(\Lambda^1 M)^{\otimes k}$ by the rule 
$$
g(\alpha_1\otimes\ldots \otimes \alpha_k,\beta_1\otimes\ldots\otimes \beta_k)=g(\alpha_1^\sharp,\beta_1^\sharp)\ldots g(\alpha_k^\sharp,\beta_k^\sharp).
$$
The space $\Lambda^kM$  is canonically embedded into $(\Lambda^1M )^{\otimes{k}}$ as the subspace  of alternating tensors. Therefore, we can consider $\Lambda^{p,q}M$ as a subspace of $(\Lambda^{1}M)^{\otimes{p+q}}$
Hence, $g$ induces a metric on $\Lambda^{p,q}M$.  

For decomposable elements of $\Lambda^{*,*}M$, set
$$(\alpha_1\otimes\beta_1)\wedge (\alpha_2\otimes\beta_2) :=(\alpha_1\wedge\alpha_2)\otimes(\beta_1\wedge\beta_2).
$$
This extends uniquely by $C^\infty(M)$-bilinearity to a product on $\Lambda^{*,*}M$.

Define an operator $$\star: \Lambda^{p,q} M\to \Lambda^{n-p,n-q}M.$$
by
$$
\star\left(\sum_i\alpha_i\otimes\beta_i\right)=\sum_i *\alpha_i\otimes *\beta_i,
$$
where $*$ is the Hodge star. Since $*$ is linear over smooth functions, the operator $\star$ is uniquely  defined.
\begin{proposition}
    The operator $\star$ is the unique $C^\infty(M)$-linear operator 
    $
\Lambda^{p,q} M\to \Lambda^{n-p,n-q}M
    $
    such that for any $\omega,\eta\in\Lambda^{p,q}M$ we have 
    $$
    \omega\wedge \star\eta=g(\omega,\eta)\operatorname{vol}_g\otimes \operatorname{vol}_g.
    $$
\end{proposition}
\begin{proof}
    If two operators $\star$ and $\star'$ satisfy the required condition, then for any $\omega,\eta\in\Lambda^{p,q}M$ we have $\omega\wedge (\star-\star')\eta=0$. Therefore, $\star=\star'$.

  It is enough to check that $\star$ satisfies the identity at each  fiber $\Lambda_x^{p,q}M$, ${x\in M}$. Let $\alpha_1,\ldots, \alpha_n$ be an orthonormal basis of $\Lambda^1_x M$. Then
    $$
    {\{\alpha_{I}\otimes \alpha_J\}_{|I|=p,|J|=q} }
    $$
    is an orthonormal basis of $\Lambda_x^{p,q}M, x\in M$. In this basis, $\star$ is expressed as
    $$
    \star(\alpha_{I}\otimes \alpha_J)={\operatorname{sgn}(I,I')\operatorname{sgn}(J,J')}\alpha_{I'}\otimes \alpha_{J'}
    $$
    where $I'$ and $J'$ are complements of $I$ and 
$J$ in $(1,\ldots ,n)$. Then it is easy to see that  $\star$  satisfies the required condition.
\end{proof}

Denote 
$$
K=\Lambda^n M.
$$
Define two $\C^\infty(M)$-linear isomorphisms 
$
\kappa: \Lambda^{p,q}M\to \Lambda^{p,q}M\otimes K^*
$
and ${\mu: \Lambda^{p,n}M\otimes K^*\to \Lambda^pM}$ by
$$
\kappa(\omega)=\omega \otimes \text{vol}_g^*, 
$$
$$
\ \ \ \mu ((\alpha\otimes \text{vol}_g)\otimes \text{vol}_g^*)=\alpha.
$$
for any $\omega\in \Lambda^{p,q}M$, $\alpha\in \Lambda^pM$. 

Denote
$$
\tilde \star=\kappa \circ\star: \Lambda^{p,q}M\to \Lambda^{n-p,n-q}M\otimes K^*.
$$

\begin{proposition}[\cite{Sh2},\cite{Sh3}]
    For any $\omega,\eta \in \Lambda^{p,q}M$ we have:
    $$
    g_{L^2}(\omega, \eta)=\int_M \mu(\omega\wedge \tilde \star\eta).
    $$
\end{proposition}
    

\begin{proposition}[\cite{Sh2},\cite{Sh3}]
   Let $\d^*$ and $\bar \d ^*$ be the adjoint operators  to $\d$ and $\bar \d$. On the space $\Lambda^{p,q}M$ we have
    $$
    \d^*=(-1)^p\tilde\star^{-1}\d \tilde\star, \ \ \ \ \ \ \bar \d^*=(-1)^q\tilde\star^{-1}\bar \d \tilde\star.
    $$
    
\end{proposition}

\begin{defin}
    Define the Laplacians $\square$  and  $\bar \square$ with respect to  $\d$ and $\bar \d$ by
    $$
    \square=\d \d^*+ \d^*\d, \ \ \ \bar \square=\bar\d \bar\d^*+ \bar \d^*\bar \d 
    $$
    A form $\alpha\in \Lambda^{p,q}M$ is called $\square$-harmonic if $\square \alpha =0$. The set of all $\square$-harmonic forms is denoted as $H_\square^{p,q}(M)$. Denote by  $P^{q}(M)$ the set of flat forms in $\Lambda^{q}(M)$.  
\end{defin}
\begin{lemma}\label{lemma26}
    For any $\alpha\in \Lambda^pM,\beta\in P^qM$ we have 
$$
\square(\alpha\otimes \beta)=(\Delta\alpha)\otimes \beta
$$     
\end{lemma}
\begin{proof}
    For any $\alpha \in \Lambda^pM,\beta\in P^qM$ we have
$$
\d( \alpha \otimes \beta) = e ((D\alpha)\otimes \beta)=d\alpha\otimes \beta.
$$
and for any $\alpha_1,\alpha_2 \in \Lambda^pM,\beta_1,\beta_2\in P^qM$ we have
    \begin{multline*}
           g( d \alpha_1 \otimes \beta_1,\alpha_2\otimes \beta_2)=g(d\alpha_1,\alpha)g(\beta_1,\beta_2) \\
   =g(\alpha_1,d^*\alpha)g(\beta_1,\beta_2) = g(  \alpha_1 \otimes \beta_1,d^*\alpha_2\otimes \beta_2).
    \end{multline*}
  Therefore, for any $\alpha \in \Lambda^pM,\beta\in P^qM$ we have
    $$
\d^* ( \alpha \otimes \beta) =d^* \alpha \otimes \beta. 
    $$
   Thus, for any $\alpha \in \Lambda^pM,\beta\in P^qM$ we have
    $$
\square ( \alpha \otimes \beta)=(\d\d^*+\d^*\d) ( \alpha \otimes \beta)=(dd^*+d^*d) \alpha \otimes \beta=\Delta\alpha\otimes\beta
    $$
\end{proof}

Combining Lemma \ref{lemma26} with the sheaves equality $\Lambda^{p,q}M=\Lambda^pM\otimes P^qM$, we get the following two propositions.  
\begin{proposition}
  We have ${\cal H}^{p,q}_\square(M)={\cal H}_\Delta^p(M,P^q)$.    
\end{proposition}
\begin{proposition}
        The operator $\square : \Lambda^{p,q}M\to \Lambda^{p,q}M$ is elliptic.
    \end{proposition}
    
    The following corollary follows from Spectral Theorem for elliptic operators on compact manifolds and Ellitic Regularity Theorem, as in the complex case. 
    \begin{cor}\label{Elliptic}
    Consider the operator $\square$ as an operator on the Hilbert space $L^2(\Lambda^{p,q}M)$. Then the following conditions are satisfied:
    \begin{itemize}
        \item [a)] There is an orthonormal basis in the Hilbert space $L^2(\Lambda^{p,q}M)$ consisting
of eigenvectors of $\square$.
\item[b)] The space of harmonic forms ${\cal H}_\square^{p,q} M$ is finite dimensional. 
\item[c)]  All eigenvectors  of $\square$ are analytic.
\item[d)] We have $L^2(\Lambda^{p,q}M)={\cal H}_\square^{p,q}(M)\oplus \square(L^2(\Lambda^{p,q}M))$.
     \end{itemize}
     
    \end{cor}
\begin{theorem}[\cite{Sh2},\cite{Sh3}]\label{cohomology}
    We have the following canonical isomorphisms
    $$
    H^{p,q}(M)=H^{p,q}_\square(M).
    $$
\end{theorem}
\begin{proof}
    The theorem follows from Corollary \ref{Elliptic} by the same way as in the classical case of Dolbeault cohomology. 
\end{proof}
\section{K\"unneth formula for $L^2$-forms}

    \begin{theorem}\label{L2Kunneth}
Let $M$ and $N$ be two compact flat affine manifolds. Then indecomposable forms are dense in the space of all forms $\Lambda^{p,q}(M\times N)$. In other words, we have 
$$
L^2(\Lambda^{p,q}(M\times N))=\bigoplus_{j\le p,j\le q} L^2(\Lambda^{i,j}M)\hat\otimes L^2(\Lambda^{p-i,q-j}(N))
$$
    \end{theorem}
    \begin{proof}
    We can consider
    $
    \bigoplus_{j\le p,j\le q} L^2(\Lambda^{i,j}M)\hat\otimes L^2(\Lambda^{p-i,q-j}(N))$ as a subset of $L^2(\Lambda^{p,q}(M\times N))$.
        Hence, it is enough to prove that if $\eta\in \Lambda^{p,q}M$ is a differential form such that for any $\alpha\in \Lambda^{i,j} M, \beta\in \Lambda^{p-i,q-j}M$ we have 
        $$
        \int_{M\times N} g(\eta,(\alpha\otimes\beta))\text{Vol}_g=0
        $$
        then $\eta=0.$ Suppose for a point $(x,y)\in M\times N$ we have $\eta_{(x,y)}\ne 0$. Then there exists  a neighborhood $U\ni(x,y)$ and forms $\alpha\in \Lambda^{i,j} M, \beta\in \Lambda^{p-i,q-j}M$  with compact support lying in $U$ such that 
        $$
        g(\eta,(\alpha\otimes\beta)) \ge 0, \ \ \ \text{and} \ \ \ \eta_{(x,y)}=(\alpha\otimes\beta)_{(x,y)}.
        $$
        Then we have 
        $$
        g(\eta,(\alpha\otimes\beta)) > 0,
        $$
        in a neighborhood of $(x,y)$. Hence,
        $$
        \int_{M\times N} g(\eta,(\alpha\otimes\beta))\text{Vol}_g>0.
        $$
        Thus, for any $(x,y)\in M$ we have $\eta_{(x,y)}=0$.
    \end{proof}

\begin{lemma}
    Let $M,N$ be two flat affine manifolds. Then we have $$\square_{M\times N}=\square_M+\square_N$$ i.e. for any $\alpha\in \Lambda^{i,j} M, \beta\in \Lambda^{p-i,q-j}N$ we have
    $$
    \square_{M\times N}(\alpha\otimes \beta)=(\square_{M} \alpha)\otimes \beta+\alpha\otimes (\square_N \beta).
    $$
\end{lemma}\label{SumOfLaplacians}
It follows from the identity $\d^*=(-1)^p\tilde\star^{-1}\d \tilde\star$ that
\begin{multline*}
    \d_M\d_N^*(\alpha\otimes\beta)+\d_N^*\d_M(\alpha\otimes\beta)
        \\=(-1)^p\d_M  \alpha\otimes (\tilde\star^{-1}\d_N \tilde\star \beta)+(-1)^{p+1}\d_M  \alpha\otimes (\tilde\star^{-1}\d_N \tilde\star \beta)=0.
\end{multline*}
That is, $\d_N^*\d_M+\d_M \d_N^*=0$. Similarly, we have $\d_M^*\d_N+\d_N \d_M^*=0$. Thus
\begin{multline*}
\square_{M\times N}=(\d_M+\d_N)(\d_M^*+d_N^*)+(\d_M^*+d_N^*)(\d_M+\d_N)\\=
\d_M\d_M^*+\d_M^*\d_M + \d_N\d_N^*+\d_N^*\d_N=\square_M+\square_N.
\end{multline*}

	\begin{theorem}\label{main}
        Let $M$ and $N$ be compact flat affine manifolds. Then we have 
        $$
        H^{p,q}(M\times N)=\bigoplus_{0\le i\le p, 0\le j\le q} H^{i,j}(M)\otimes H^{p-i,q-j}(N).
        $$
	\end{theorem}
    \begin{proof}
        According to Theorem \ref{cohomology} it is enough to prove that  
          $$
        H^{p,q}_\square(M\times N)=\bigoplus_{1\le i\le p, 1\le j\le q} H^{i,j}_\square(M)\otimes H^{p-i,q-j}_\square(N).
        $$
        It immediately follows from Lemma \ref{SumOfLaplacians}, that 
        $$
        \bigoplus_{1\le i\le p, 1\le j\le q} H^{i,j}_\square(M)\otimes H^{p-i,q-j}_\square(N)\subset H^{p,q}(M,N)
        $$ 
        Let $\{\alpha_1,\alpha_2\ldots\}$ be a basis of $\bigoplus_{i,j\ge 0} L^2(\Lambda^{i,j}M)$ consisting of eigenvectors of $\square_M$
        and $\{\beta_1,\beta_2\ldots\}$ be a basis of $\bigoplus_{i,j\ge 0} L^2(\Lambda^{i,j}M)$ consisting of eigenvectors of $\square_N$. According to Theorem \ref{L2Kunneth} $\{\alpha_k\otimes \beta_l\}$ is a basis of $\bigoplus_{i,j\ge 0} L^2(\Lambda^{i,j}(M\times N))$  consisting of eigenvectors of $\square_{M\times N}$. Let 
        $$
        \square_M\alpha_k=a_k\alpha_k, \ \ \  \square_N\beta_l=b_l\alpha_l
        $$
        According to Lemma \ref{SumOfLaplacians}, we have 
        $$
        \square_{M\times N} (\alpha_k\otimes \beta_l)=(a_k+b_l)(\alpha_k\otimes \beta_l)
        $$
        Since for any $\omega \in \Lambda^{p,q}M$ we have $g(\d \d^*+\d^*\d \omega,\omega)=|\d\omega|^2+|\d^*\omega|^2$, all eigenvalues $\{a_k\},\{b_l\}$ are positive. Therefore, a form
        $$
        \eta=\sum_{k,l\ge 0} c_{k,l} \alpha_k\otimes \beta_l 
        $$
        is harmonic if and only if for all nonzero $c_{k,l}$ the forms $\alpha_k$ and $\beta_l$ are harmonic. Thus, any harmonic form $\eta\in  H^{p,q}_\square(M\times N)$ lies in the space $$\bigoplus_{1\le i\le p, 1\le j\le q} H^{i,j}_\square(M)\otimes H^{p-i,q-j}_\square(N).$$
    \end{proof}
		\section{Koszul cohomology and Hessian manifolds}
		\begin{defin}
			Let $(M,D)$ be a flat affine manifold. A {\bfseries Hessian} metric on $M$ is a Riemannian metric $g$ such that $Dg$ is a totally symmetric tensor. Equivalently, a Riemannian metric $g$ is Hessian if and only if it is locally expressed as $D\alpha$ where $\alpha$ is a locally defined 1-form. A Hessian metric $g$ is said to be of {\bfseries Koszul type} if it  can be expressed as $g=D\alpha$ where $\alpha$ is a globally defined $1$-form.  
			A Hessian manifold $(M,D,g)$ is a flat affine manifold endowed with a Hessian metric $g$.   
		\end{defin}
		  There is also a coordinate condition: a metric $g=g_{ij} dx^i dx^j$ is Hessian if and if for any $i,j,k$ we have
          $$
          \frac{\d g_{ij}}{\d x^k}-\frac{\d g_{jk}}{\d x^i}=0
          $$
          (\cite{Sh3}).
		\begin{defin}
			A flat affine manifold $M$ is called {\bfseries hyperbolic} if and only if $M=U/\Gamma$ where $U\subset\R^n$ is a convex domain without straight lines and $\Gamma$ is a discrete subgroup of affine automorphisms of $U$.  
		\end{defin}

		\begin{theorem}[\cite{K}]\label{HessianMetricsOnHyperbolicManifolds}
			A compact flat affine manifold $M$ admits a Hessian metric of Koszul type $g=D\alpha$ if and only if $M$ is hyperbolic. 
		\end{theorem}
        	\begin{theorem}[\cite{K}] \label{CohomologyOfHyperbolicManifolds}
		    Let $M$ be a compact Hessian manifold. Then $H^{1,1}(M)=0$ if and only if $M$ is hyperbolic. 
		\end{theorem}
		\begin{proposition}
			Let  $(M,D)$ be a flat affine manifold, $\alpha\in \Lambda^1 M$  a closed 1-form and $g\in \Lambda^1M\otimes\Lambda^1M=\Lambda^{1,1}M$  a Riemannian metric.
            \begin{itemize}
                \item[a)] For any $1$-form $\alpha$ we have $\d (1\otimes \alpha)=D\alpha$. Hence, $g$ is a Hessian metric of Koszul type if and only if $g$ is $\d$ exact.
                \item[b)] The metric $g$ is Hessian if and only if $\d g=0.$ 
            \end{itemize}

		\end{proposition}
        \begin{proof}
            The identity $\d (1\otimes \alpha)=D\alpha$ immediately follows from the definition of $\d$. Let $x^1,\ldots,x^n$ be local flat coordinates  and $g=g_{ij}dx^i dx^j$. Then
            \begin{multline*}
                \d g=\d\left( g_{ij}dx^i \otimes dx^j\right)=\frac{\d g_{ij}}{\d x^k} \left(dx^k\wedge dx^i\right)\otimes dx^j \\ =\frac{\d g_{ij}}{\d x^k} \left(dx^k\otimes dx^i-dx^i\otimes dx^k\right)\otimes dx^j =\left(\frac{\d g_{ij}}{\d x^k}-\frac{\d g_{jk}}{\d x^i}\right)dx^k\otimes dx^i\otimes dx^j
            \end{multline*}
         Thus, the metric $g$ is Hessian if and only if $\d g=0.$

        \end{proof}

We  study the set of Hessian metrics on a flat affine manifold. Any Hessian metric $g$ determines a cohomology class $[g]\in H^{1,1}(M)$. For any closed 1-form $\alpha$ such that $g+D\alpha$ is positively definite at every point, the tensor $g+d\alpha $ is again a Hessian metric. Hence, the description of all Hessian metrics reduces to calculating $H^{1,1}(M)$ and characterizing the set of all Hessian classes within it.

		\section{Applications to Hessian metrics} 
        
	 	 For simplicity, we implicitly identify forms on $M$ and $N$ with their pullbacks to $M\times N$.
			\begin{theorem}\label{HessianMetricsOnproducctsOfManifolds}
			Let $M$ and $N$ be compact Hessian manifolds. Then any Hessian metric $g$ on  $M\times N$ is equal to
			$$
			g=g_M+g_N+D \alpha
			$$
			where $g_M, g_N$ are Hessian metrics on $M,N$ and $\alpha$ is a closed 1-form on $M\times N$.
		\end{theorem}
       
		\begin{proof}
			According to K\"unneth formula, we have 
			\begin{multline*}
			H^{1,1}(M\times N)\\ =H^{1,1}(M)\oplus  H^{1,1}(N) \oplus \left(H^{1,0}(M)\otimes H^{0,1}(N)\right)\oplus \left(H^{0,1}(M)\otimes H^{1,0}(N)\right)
			\end{multline*}
			
			Therefore, any Hessian metric $g$ admits a form 
			$$
			g=g_M+ g_N +\eta_M\otimes \theta_N +\theta_N'\otimes \eta_M'+ d\alpha,
			$$
            where $g_M,g_N$ are Hessian metrics on $M$ and $N$, respectively; $\eta_M,\eta_M'$ are 1-forms on $M$; $\theta_N,\theta_N'$ is 1-forms on $N$; and $\alpha$ is a closed 1-form on $M\times N$.
            
  			Since the bilinear forms $g,	g_M, g_N,$ and $d\alpha $ are symmetric, the bilinear form $\eta_M\otimes \theta_N +\theta_N'\otimes \eta_M'$ is also symmetric. Hence, 
			$$
			\eta_M=\eta_M, \qquad \theta_N=\theta_N'.
			$$
			Since the tensors $Dg, Dg_M,Dg_N,$ and $D d\alpha$ are totally symmetric, the tensor $D(\eta_M\otimes \theta_N +\theta_N\otimes \eta_M)$ is also totally symmetric.

			Let $X,Y$ be locally defined flat vector fields on $M$ and $Z$ be a locally defined flat vector field on $N$. Then 
			$$
			D(\eta_M\otimes \theta_N +\theta_N\otimes \eta_M)(X,Y,Z)=Z\left((\eta_M\otimes \theta_N +\theta_N\otimes \eta_M)(X,Y)\right)=0.
			$$  
			For the same reason,  for any locally defined flat vector field $X$ on $M$ and any locally defined flat vector fields $Y,Z$ on $N$ we have 
			$$
				D(\eta_M\otimes \theta_N +\theta_N\otimes \eta_M)(X,Y,Z)=0.
			$$
			Also, if $X,Y,Z$ are all on $M$ or all on $N$ then
			$$
				D(\eta_M\otimes \theta_N +\theta_N\otimes \eta_M)(X,Y,Z)=0.
			$$
			Therefore, 
			$$
			\eta_M\otimes \theta_N +\theta_N\otimes \eta_M=0
			$$
			and 
			$$
			g=g_M+g_N+D \alpha.
			$$
		\end{proof}
		
		\begin{cor}
			Let $M$ be a compact Hessian manifold and $N=C/\Gamma$ be a compact hyperbolic manifold. Then any Hessian metric $g$ on $M\times N$ is equal to
			$$
			g=g_M +D\alpha,
			$$
			where $g_M$ is a Hessian metric on $M$ and $\alpha$ is a closed 1-form on $M\times N$. 
		\end{cor}
        
        \begin{proof}
            According to Theorem \ref{HessianMetricsOnproducctsOfManifolds}, we have 
            $$
            g=g_M+g_N+D \alpha,
            $$
            where $g_M, g_N$ are Hessian metrics on $M,N$ and $\alpha$ is a closed 1-form on $M\times N$. According to Theorem \ref{CohomologyOfHyperbolicManifolds}, there exist a closed 1-form $\beta$ on $N$ such that $g_N=D\beta$. Therefore,
            $$
            g=g_M+D\alpha',
            $$
            where $\alpha'=\alpha+\beta$.
        \end{proof}
		\begin{theorem}\label{TorusCohomoly}
			Let $T^n=\R^n/\Z^n$ be a flat torus. Then
			$$
			H^{**}(T^n)=\R[a^1,\ldots,a^n,b^1,\ldots,b^1n]/_{\left(a^1\right)^2,\ldots, \left(b^n\right)^2}.
			$$
			where $a^i=[dx^i\otimes 1]\in H^{1,0} (T^n), b^i=[1\otimes dx^i]\in H^{0,1} (T^n)$.
		\end{theorem}

        \begin{proof}
            According to Theorem \ref{main}, it is enough to prove that $$
H^{**}(T^1)=\R [a,b]/_{a^2,b^2},
            $$
         where $a=[dx\otimes1]$ and $b=[1\otimes dx]$.
            
We have 
$$
\d(\alpha\otimes 1)=(e\circ D(\alpha))\otimes 1=d\alpha\otimes 1.
$$
Therefore, the correspondence 
$
[\alpha] \mapsto[\alpha\otimes 1]
$
sets an isomorphism  $$H^{p}(T^1)\simeq H^{p,0}(T^1).$$ In particular, we get that
$$
{H^{0,0}(T^1)=\R[1\otimes 1]} \qquad \text{and} \qquad H^{1,0}(T^1)=\R[dx\otimes 1].
$$
            
            Any element of $1\otimes \alpha\in H^{0,1}(T^1)=\operatorname{ker} \left( \d :\Lambda^{0,1}T^1\to \Lambda^{1,1}T^1\right)$ admits a form $1\otimes \alpha$.
            We have 
            $$
            \d((1\otimes \alpha))= e D(1\otimes \alpha)=0
            $$
            if and only if $D\alpha=0$. The space of flat 1-forms on $T^1$ is 1-dimensional and generated by $dx$, that is $$
            H^{0,1}(T^1)=\R[1\otimes dx].
            $$

Any element of $H^{0,1}(T^1)$ admits a form $f\otimes dx$, where $f\in C^\infty(T^1)$; Any element of $H^{1,1}(T^1)$ admits a form $\alpha\otimes dx$, where $\alpha\in \Lambda^1T^1$.
 For any smooth function $f$ we have
   $$
   \d ((f\otimes dx))=e\circ (D(f)\otimes dx)=df\otimes dx.
   $$
Therefore, the correspondence 
$
[\alpha] \mapsto[\alpha\otimes dx]
$
sets an isomorphism  $$H^{p}(T^1)\simeq H^{p,1}(T^1).$$ In particular, we get that
$$
{H^{0,1}(T^1)=\R[1\otimes dx]} \qquad \text{and} \qquad H^{1,1}(M)=\R[dx\otimes dx].
$$

            We have proved that $H^{**}(T^1)=\R [a,b]/_{a^2,b^2}$, where $a=[dx\otimes 1]$ and $b=[1\otimes dx]$ are generators of $H^{1,0}(T^1)$ and $H^{0,1}(T^1)$.

        \end{proof}

			
			\begin{theorem}
				Let $(M,D)$ be a compact flat affine manifold that admits a $D$-flat Riemannian metric. Then any Hessian metric on $M$ is equal to 
				$$
				g=g_{\text{flat}} +D\alpha
				$$
				where $g_{\text{flat}}$ is a flat Riemannian metric on $M$ and $\alpha $ is a closed 1-form.
			\end{theorem}
            \begin{proof}
                According to Bieberbach theorem, there is a finite covering $\pi: T^n\to M$, where $T^n$ is a torus. In other words, we have $M=T^n/\Gamma$, where $\Gamma$ is discrete subgroup of affine automorphisms. According to Theorem \ref{TorusCohomoly}, the cohomology class $[\pi^*g]\in H^{1,1}(T^n)$ can be represented by a flat tensor $\tilde g_{\text{flat}}$. That is, we have
                $$
                \pi^*g= \tilde g_{\text{flat}}+ D\tilde\alpha, 
                $$
                where $\tilde\alpha$ is a closed 1-form on $T^n$. Since the tensors $\pi^*g$ and $D\tilde \alpha$ symmetric, the tensor $\tilde g_{\text{flat}}$ is symmetric, too. Let us check that $\tilde g_{\text{flat}}$ is positive definite. It is enough to check that for any nonzero flat vector field $X$, we have ${\tilde g_\text{flat}(X,X)>0}$. Let the function $\tilde \alpha(X)$ reach a maximum at a point $p\in T^n$. Then 
                $$
                D\tilde\alpha(X,X)|_p=X(\tilde\alpha(X))|_p=0.
                $$
                Hence
                $$
                \tilde g_{\text{flat}}(X,X)=\tilde g_{\text{flat}}(X,X)|_p=\pi^*g(X,X)|_p+D\tilde\alpha(X,X)|_p=\pi^*g(X,X)|_p>0.
                $$
                We have proved that $g_{\text{flat}}$ is Riemannian metric. Define the averages
                $$
                \hat g_{\text{flat}}=\frac{1}{|\Gamma|}\sum_{\gamma\in \Gamma}\gamma^*\tilde g_{\text{flat}} \qquad  \text{and} \qquad\hat \alpha=\frac{1}{|\Gamma|}\sum_{\gamma\in \Gamma}\gamma^*\tilde\alpha. 
                $$
                Since $\pi^* g$ is $\Gamma$-invariant we have 
                  $
                  \pi^*g= \hat g_{\text{flat}}+ D\hat\alpha.
                  $
                  Then we can push down  $\hat g_{\text{flat}}$ and $\hat\alpha$ to $M$  and obtain a flat metric $g_\text{flat}$ and a close 1-form $\alpha$ on $M$ such that 
                  $
                  g=g_{\text{flat}}+ D\alpha.
                  $
            \end{proof}

\end{document}